\documentclass[12pt,twoside]{article}
\usepackage[left=2.5cm,right=2.5cm]{geometry}
\usepackage{amsmath,amsthm,amscd,amssymb,amsfonts}
\usepackage{latexsym,graphics,longtable,lscape}
\usepackage{mathtools, mathrsfs}	
\usepackage{epstopdf}
\usepackage{float}
\usepackage{indentfirst}
\usepackage{cite}
\usepackage{verbatim}
\usepackage{mleftright}
\usepackage{enumitem}
\renewcommand{\baselinestretch}{1}

\usepackage{color}
\definecolor{myblue}{rgb}{0.13,0.29,0.56}

\usepackage{varioref}
\usepackage{nameref}
\usepackage[CJKbookmarks=true,colorlinks,linkcolor=black,anchorcolor=black,citecolor=black]{hyperref}
\usepackage[nameinlink]{cleveref}

\crefname{section}{\color{black} Section}{\color{black} Section}
\crefname{lemma}{Lemma}{Lemmas}
\crefname{prob}{Problem}{Problems}
\crefname{coro}{Corollary}{Corollaries}
\crefname{conj}{Conjecture}{Conjectures}
\crefname{prop}{Proposition}{Propositions}
\crefname{claim}{Claim}{Claims}
\crefname{defn}{Definition}{Definitions}
\crefname{rmk}{Remark}{Remarks}
\crefname{exm}{Example}{Examples}
\crefname{thm}{Theorem}{Theorems}

\def\br{\mathbb{R}}
\def\bn{\mathbb{N}}

\theoremstyle{plain}
\newtheorem{thm}{Theorem}[section]
\newtheorem{lemma}[thm]{Lemma}

\newtheorem{coro}[thm]{Corollary}

\theoremstyle{definition}

\newtheorem{rmk}[thm]{Remark}

\newcommand{\bremark}{\begin{remark} \em}
\newcommand{\eremark}{\end{remark} }
\newenvironment{prf}{{\noindent\bf Proof.}}{\hfill $\square$\par}

\numberwithin{equation}{section}

\begin{document}
\parindent 15pt
\renewcommand{\theequation}{\thesection.\arabic{equation}}
\renewcommand{\baselinestretch}{1.15}
\renewcommand{\arraystretch}{1.1}
\renewcommand{\vec}[1]{\bm{#1}}
\def\disp{\displaystyle}
\title{\bf\large Symmetry of Convex Solutions to Fully Nonlinear Elliptic Systems: Bounded Domains\thanks{Supported by National Natural Science Foundation of China (11771428, ????????)}

\author{{\small Weijun Zhang$^{1}$\thanks{Email: zhangweij3@mail.sysu.edu.cn}~~~Zhitao Zhang$^{2,3,4}$\thanks{Corresponding author. Email: zzt@math.ac.cn}}\\
{\small $^1$ School of Mathematics, Sun Yat-sen University, Guangzhou 510275, People's Republic of China,}\\
{\small $^2$ Academy of Mathematics and Systems Science,}\\
{\small Chinese Academy of Sciences, Beijing 100190, People's Republic of China,}\\
{\small $^3$ School of Mathematical Sciences, University of Chinese}\\
{\small Academy of Sciences, Beijing 100049, People's Republic of China,}\\
{\small $^4$ School of Mathematical Science, Jiangsu University, Zhenjiang 212013, People's Republic of China.}}
}

\date{}

\maketitle

\abstract
{\small In this paper, we are concerned with the monotonic and symmetric properties of convex solutions to fully nonlinear elliptic systems. We mainly discuss Monge-Amp\`ere type systems for instance, considering
	\begin{equation*}
		\det(D^2u^i)=f^i(x,{\bf u},\nabla u^i), \ 1\leq i\leq m,
	\end{equation*}
over bounded domains of various cases, including the bounded smooth simply connected domains and bounded tube shape domains in $\br^n$. We obtain monotonic and symmetric properties of the solutions to the problem with respect to the geometry of domains and the monotonic and symmetric properties of right-hand side terms. The proof is based on carefully using the moving plane method together with various maximum principles and Hopf's lemmas. The existence and uniqueness to an interesting example of such system is also discussed as an application of our results.
}

\section{Introduction}
\noindent

In this paper, we consider the following Monge-Amp\`ere systems:
\begin{equation}\label{eq:MA}
\det(D^2 u^i(x))=f^i(x,{\bf u}(x),\nabla u^i(x)),\ in\ \Omega,\ 1\leq i\leq m,
\end{equation}
where $\Omega\subset\br^n$, ${\bf u}=(u^1,\dots,u^m)$, and ${\bf f}=(f^1,\dots,f^m)$, $n,m\in\bn^*$, satisfy some suitable conditions in different cases.

\subsection{Background}

The monotonicity and symmetry properties are very useful in the research on nonlinear partial differential equations and has attracted much attention in many areas of mathematics. A powerful tool for studying these is the method of moving plane, especially when the equations may have no variational structure, see \cite{chen_methods_2010, chen_moving_2003, gidas_symmetry_1979, gui_sphere_2018, zhang_variational_2013} for examples. The method of moving plane originally discovered by Alexandrov \cite{alexandrov_characteristic_1962}, and then have been deeply developed by Serrin \cite{serrin_symmetry_1971}, Gidas-Ni-Nirenberg\cite{gidas_symmetry_1979, gidas_symmetry_1981}.

After that, symmetry properties in the case of a single equation have received considerable investigation by various authors. For example, in the  1990s, Li established monotonicity and symmetry results of solutions to single fully nonlinear elliptic equations on bounded domains in \cite{li_monotonicity_1991} and unbounded domains in \cite{li_monotonicity_1991-1}, respectively. Zhang and Wang \cite{zhang_existence_2009} consider the Monge-Amp\`ere equation with exponential right-hand side term, which arising from the differential geometry problem, in arbitrary convex domains. 

However, to the best of our knowledge, the study for the case of systems are much less than scalar case. The first work dates back to Troy \cite{troy_symmetry_1981}. Later ones are De Figueiredo\cite{de_figueiredo_monotonicity_1994} and Busca \cite{busca_symmetry_2000}, in where symmetry results were obtained for elliptic systems in the general domains and the whole space, respectively. Recently, Ma and Liu \cite{ma_symmetry_2010,ma_symmetry_2010-1,liu_symmetry_2012,liu_symmetry_2013} treated different systems over various domains, among which \cite{ma_symmetry_2010} concerns the Monge-Amp\`ere systems arising from the differential geometry problem. It is of great interest to further investigate the symmetry properties of Monge-Amp\`ere systems based on the aforementioned works. 

The goal of this paper is to give a rather complete and general version of monotonicity and symmetry results to the Monge-Amp\`ere system over bounded domains of various cases, including the bounded smooth simply connected domains and bounded tubes shape domains in $\br^n$. The cases of unbounded domains can be seen here \cite{zhang2024_1}.

\subsection{Main Results}

The main results about symmetry are formally stated as below, in fact we get a more general results about monotonicity, more detail could be seen in \cref{sec:bdd,sec:bt}.

In order to state our main results, we need firstly introduce some basic hypotheses on$f^i:\overline{\Omega}\times\br^m\times\br^n\to\br, (x,{\bf z},p)\to f^i(x,{\bf z},p)$, where ${\bf z}=(z^1,\dots,z^m)$. We suppose that for all $1\leq i\leq m$, $f^i\in C(\overline{\Omega}\times\br^m\times\br^n,\br)$, furthermore satisfying some of the following in different situations.

In order to assure the ellipticity of \eqref{eq:MA}, we need the following two kinds of positive conditions:
\begin{enumerate}[resume,label=$(F_{\arabic{enumi}})$]
	\item \label{F0} $f^i(x,{\bf z},p)>0$, $\forall (x,{\bf z},p)\in (\Omega\times\br^m\times\br^n)$;

	\item \label{Fc} $f^i(x,{\bf z},p)\geq c_f>0$, $\forall (x,{\bf z},p)\in (\Omega\times\br^m\times\br^n)$;
\end{enumerate}

\begin{rmk}
	\ref{Fc}, which is in order to assure the uniformly ellipticity of \eqref{eq:MA}, is stronger than \ref{F0}, which is merely assure the ellipticity of \eqref{eq:MA}.
\end{rmk}

Next, when $\Omega$ assume to be convex in one direction, denote as ${\bf e_1}$, we can study whether the solutions to \eqref{eq:MA} will be having some monotonicity, hence we need the following monotonicity kind conditions on ${\bf f}$:

\begin{enumerate}[resume,label=$(F_{\arabic{enumi}})$]

\item \label{Fzi}$f^i=f^{i,1}+f^{i,2}$, where $f^{i,1}$ is locally uniformly Lipschitz continuous in the component $z^i$, and $f^{i,2}$ is non-increasing in $z^i$, whenever the remaining components $z^j$, $j\neq i$, and $x,p$ fixed;
	\item \label{Fzj}$f^i$ is non-increasing in $z^j$, $j\neq i$, whenever the remaining components $z^k$, $k\neq j$, and $x,p$ fixed;
	\item \label{Fp}$f^i$ is locally uniformly Lipschitz continuous in the component $p$, whenever $x,{\bf z}$ fixed;
        \item \label{Fanti}$f^i(y_1,x',{\bf z},\bar{p})\geq f^i(x,{\bf z},p),\ \forall\ {\bf z}\in \br^m, p\in\br^n$ and $x=(x_1,x')\in\Omega$ such that $p_1\leq 0, x_1\leq 0$ with $x_1\leq y_1 \leq -x_1$, where $\bar{p}:=(-p_1,p_2,\cdots,p_n)$;

\end{enumerate}

Furthermore, when $\Omega$ assume to be symmetric in ${\bf e_1}$, we can study whether the solutions to \cref{eq:MA} will be having some symmetry along ${\bf e_1}$, we need to strengthen \ref{Fanti}.

\begin{enumerate}[resume,label=$(F_{\arabic{enumi}})$]
	\item \label{Fsym1}$f^i(x,{\bf z},p)= f^i(|x_1|,x_2,\dots,x_n,{\bf z},|p_1|,p_2,\dots,p_n),\ \forall (x,{\bf z},p)\in (\Omega\times\br^m\times\br^n)$;
\end{enumerate}

At last, when $\Omega$ assume to be symmetric in all directions, we can study whether the solutions to \cref{eq:MA} will be radially symmetry, we need to strengthen \ref{Fsym1}.

\begin{enumerate}[resume,label=$(F_{\arabic{enumi}})$]
	\item \label{Fsymall} $f^i(x,{\bf z},p)=f^i(Ox,{\bf z},O'p),\ \forall\ O, O'\in O(n),\ \forall (x,{\bf z},p)\in(\Omega\times\br^m\times\br^n)$, where $O(n)$ is the n-th order orthogonal group;
	
\end{enumerate}

For the convenience, we denote \begin{equation}\label{eq:dij}
d_{ij}(x,{\bf z},p,h):=\left\{\begin{array}{rl}
\frac{1}{h}\left(f^i(x,{\bf z}+h{\bf e_j},p)-f^i(x,{\bf z},p)\right),&i\neq j, h\neq 0\\
\frac{1}{h}\left(f^{i,2}(x,{\bf z}+h{\bf e_i},p)-f^{i,2}(x,{\bf z},p)\right),&i=j, h\neq 0,\\
0,& h=0,\\
\end{array}
\right.
\end{equation}
and $D:=\begin{pmatrix}
d_{11}&\cdots&d_{1m}\\
\vdots&\ddots&\vdots\\
d_{m1}&\cdots&d_{mm}
\end{pmatrix}.$

\begin{rmk}\label{rmk:dij}
$\operatorname{sgn}(d_{ij}\cdot h)=-\operatorname{sgn}(h)$ and $d_{ij}\leq0, \ \forall\ i,j=1,\dots,m,$ by \ref{Fzi} and \ref{Fzj}, and all are locally bounded.
\end{rmk}

\vskip 0.1in
Now we begin to state our main results about symmetry.

\vskip 0.1in
For the case of bounded smooth simply connected domains, we mainly consider the following constant-boundary Dirichlet problem for \eqref{eq:MA},
\begin{equation}\label{eq:boundzerointro}
\left\{\begin{array}{rl}
\det(D^2 u^i,\nabla u^i))\ =&f^i(x,{\bf u},\nabla u^i),\ \text{in}\ \Omega,\\
	u^i\ =&c^i,\quad\ \quad\ \quad\ \quad\ \text{on}\ \partial\Omega,\ 1\leq i\leq m.\\
\end{array}
\right.
\end{equation}
where $c^i\in\br$ are given constants.

We have the main results as follow,
\begin{thm}
	Let $\Omega=B_R$ be a arbitrary ball with radius $R$. Assume ${\bf f}$ satisfy \ref{F0}, \ref{Fzi}, \ref{Fzj}, \ref{Fp}, \ref{Fsymall}. Let ${\bf u}=(u^1,\cdots,u^m)$ be a group of $[C^2(\overline{\Omega})]^m$ strictly convex solutions to \eqref{eq:boundzerointro}, then each $u^i$ must be radially symmetric and strictly increasing respect to the center of $B_R$. 
	
	More precisely, denote the center of $B_R$ as $x^*\in\br^n$, and $r=|x-x^*|$, then for $i=1\dots,m$, each $u^i$ must be
	\begin{equation*}
	u^i(x)=u^i(r),\ \forall\ x\in B_R(x^*),
	\end{equation*}
	moreover,
	\begin{equation*}
	\frac{d u^i}{dr}(x)>0,\ \forall\ x\in B_R(x^*),
	\end{equation*}
\end{thm}

\vskip 0.1in
For the case of bounded tube shape domains, we mainly consider the following constant-boundary Dirichlet problem for \eqref{eq:MA}
\begin{equation}\label{eq:cylinderintro}
\left\{\begin{array}{rl}
\det(D^2 u^i)\ =&f^i(x,{\bf u},\nabla u^i),\ x\in\ C_H,\\
u^i\ =&c^i,\hspace{1.9cm}x\in\partial C_H,\  i=1,\dots,m.\\
\end{array}
\right.
\end{equation}
where $c^i\in\br$ are constants.

We have the main results as follow,
\begin{thm}
	Let $C_H=\Omega\times(-H, H)$ be a cylinder in $\br^n$, that is $\Omega=B_R$ be a arbitrary ball with radius $R$ in $\br^{n-1}$. Assume ${\bf f}$ satisfy \ref{Fc}, \ref{Fzi}, \ref{Fzj}, \ref{Fp}, \ref{Fsymall}. Let ${\bf u}=(u^1,\cdots,u^m)$ be a group of $[C^2(\overline{C_H})]^m$ strictly convex solutions to \eqref{eq:cylinderintro}, then each $u^i$ must be radially symmetric and strictly increasing respect to the axis crossing the center of $B_R$.
	
	More precisely, denote the center of $B_R$ as $x^*=\left((x^*)',x^*_n\right)\in\br^n$, and denote $r=|x'-(x^*)'|$, then for $i=1\dots,m$, each $u^i$ must be
	\begin{equation*}
	u^i(x)=u^i(r,x_n),\ x\in C_H,
	\end{equation*}
	moreover, 
	\begin{equation*}
	\frac{\partial u^i}{\partial r}(r,x_n)>0,\ x\in C_H.
	\end{equation*}
\end{thm}

We mainly follow the moving plane method with concrete procedures proposed by Busca \cite{busca_symmetry_2000} and developed by Ma-Liu \cite{ma_symmetry_2010}. More precisely, we always divide the proof into three steps for each case of domains: the first step is to choose the plane we are going to move and to show that the moving procedure can be started; the last two steps are to continuously move the plane toward right to its limit position. The maximum principle and Hopf's lemma are repeatedly used in these steps. 

With respect to the cases of bounded domains, we mainly improve the existing results by reducing the smoothness condition on the right-hand side $f^i$ from $C^1$ to Lipschitz continuous. The method is spiritually similar to \cite{santos_symmetry_2020}, where symmetry properties were obtained for positive solutions to certain fully nonlinear elliptic systems with $f^i$ being Lipschitz continuous.

The paper is organized as follows. In \cref{sec:pre}, we present some preliminary results for the moving plane method. \cref{sec:bdd} is devoted to the case of bounded smooth simply connected domains. \cref{sec:bt} is devoted to the case of bounded tube shape domains. At last, we apply our symmetry results to an interesting example in \cref{sec:app}.  

\section{Some Preliminaries}\label{sec:pre}
\noindent

Noting that, in what follows, we always consider the classical solutions to the problem, that is, the solutions being twice continuously differentiable up to the boundary. This is always the case if each $f^i(x,{\bf u}(x),\nabla u^i(x))$ is $C^\alpha(\overline{\Omega})$ as a function of $x$ by the standard regularity theory of Monge-Amp\`ere equation; see \cite{figalli_monge-ampere_2017,gilbarg_elliptic_2001,le_schauder_2017,cheng_regularity_1977}. And in order to assure the ellipticity of the equations, the solutions are always considered to be strictly convex.

\vskip 0.1in

Here are some notations preparing for the moving plane method.
Fixed a direction vector $\nu\in\br^n$ with $|\nu|=1$, and a real number $\lambda\in\br$, we defined the related half space
\begin{equation*}
	\Sigma_{\lambda,\nu}:=\{x\in\Omega\ |\ x\cdot\nu<\lambda\},
\end{equation*}
and the corresponding hyperplane
\begin{equation*}
	T_{\lambda,\nu}:=\{x\in\Omega\ |\ x\cdot\nu=\lambda\}.
\end{equation*}

Let $x_{\lambda,\nu}$ be the reflection of $x\in\overline{\Omega}$ through $T_{\lambda,\nu}$, that is 
\begin{equation*}
	x_{\lambda,\nu}:=x+2(\lambda-x\cdot\nu)\nu.
\end{equation*}

correspondingly, for any set $A\subset\br^n$, let $A^\nu_\lambda$ be the reflection through $T_{\lambda,\nu}$, that is
\begin{equation*}
	A^\nu_\lambda:=\{x_{\lambda,\nu}=x+2(\lambda-x\cdot\nu)\nu\ |\ x\in A\}.
\end{equation*}

We denote that for a invertible matrix $M$, $M^{jk}:=(M^{-1})_{jk}$, and for two matrices $M_1,M_2$, denoted the Frobenius inner product as  $$\langle M_1,M_2\rangle_F:=\sum\limits_{j,k=1}^n(M_1)_{jk}(M_2)_{jk}=\operatorname{tr}\left(M_1^TM_2\right),$$
especially, if one of them is symmetric, then $\langle M_1,M_2\rangle_F=\operatorname{tr}\left(M_1M_2\right)$.

For a function $u\in C^2(\overline{\Omega})$, we define the reflected function $u_{\lambda,\nu}(x)$ through $T_{\lambda,\nu}$ as follow,
\begin{equation*}
u_{\lambda,\nu}(x):=u(x_{\lambda,\nu})=u(x+2(\lambda-x\cdot\nu)\nu),
\end{equation*}
and we have some facts that
\begin{equation*}
\frac{\partial u_{\lambda,\nu}}{\partial x_j}(x)=\sum\limits_{i=1}^n\frac{\partial u}{\partial x_i}(x_{\lambda,\nu})(\delta_{ij}-2\nu_i\nu_j)=\nabla u(x_{\lambda,\nu})\cdot \mu^{\nu}_j,
\end{equation*}
where $\mu^{\nu}_j:=(-2\nu_1\nu_j,\cdots,1-2\nu_j^2,\cdots,-2\nu_n\nu_j),$
and 
\begin{equation*}
\frac{\partial^2 u_{\lambda,\nu}}{\partial x_k\partial x_j}(x)=\nabla (\frac{\partial u_{\lambda,\nu}}{\partial x_j}(x))\cdot \mu^{\nu}_j=\nabla (\nabla u(x_{\lambda,\nu})\cdot \mu^{\nu}_k)\cdot \mu^{\nu}_j=\langle D^2(x_{\lambda,\nu}),(\mu^{\nu}_k)^T\mu^{\nu}_j\rangle_F,
\end{equation*}
thus
\begin{equation*}
\nabla u_{\lambda,\nu}(x)=\nabla u(x_{\lambda,\nu})\cdot
\begin{pmatrix}
1-2\nu_1^2 & -2\nu_1\nu_2 & \cdots & -2\nu_1\nu_n \\
-2\nu_2\nu_1 & 1-2\nu_2^2 & \cdots & -2\nu_2\nu_n \\
\vdots & \vdots & \ddots & \vdots \\
-2\nu_n\nu_1 & -2\nu_n\nu_2 & \cdots & 1-2\nu_n^2
\end{pmatrix}
=\nabla u(x_{\lambda,\nu})\cdot(I-2\nu^T\nu),
\end{equation*}

And we at last define the difference function $U_\lambda(x)$,
\begin{equation*}
U_{\lambda,\nu}(x):=u_{\lambda,\nu}(x)-u(x).
\end{equation*}

Once if the domain is somehow convex in one direction, for example $\nu={\bf e_1}=(1,0,\dots,0)\in\br^n$, in this case, for shortly, we denote
\begin{align*}
	&T_\lambda:=T_{\lambda,e_1}=\{x\in\Omega\ |\ x_1=\lambda\},\\ &\Sigma_\lambda:=\Sigma_{\lambda,e_1}=\{x\in\Omega\ |\ x_1<\lambda\},\\ 
	&x_\lambda:=x_{\lambda,e_1}=(2\lambda-x_1,x'),\ where\  x'=(x_2,\dots,x_n)\in\br^{n-1}\\ &u_\lambda(x):=u_{\lambda,e_1}(x)=u(2\lambda-x_1,x'). 
\end{align*} 

We can easily see that $$\nabla u_\lambda(x)=\left(-\frac{\partial u}{\partial x_1}(x_\lambda),\frac{\partial u}{\partial x_2}(x_\lambda),\dots,\frac{\partial u}{\partial x_n}(x_\lambda)\right)=\nabla u(x_\lambda)\cdot \bar{D},$$ and the Hessian matrix of $u_\lambda$ is
\begin{equation*}
D^2u_\lambda(x)=
\begin{pmatrix}
	\frac{\partial^2 u}{\partial x_1^2}(x_\lambda) & -\frac{\partial^2 u}{\partial x_1\partial x_2}(x_\lambda) & \cdots & -\frac{\partial^2 u}{\partial x_1\partial x_n}(x_\lambda) \\
	-\frac{\partial^2 u}{\partial x_2\partial x_1}(x_\lambda) & \frac{\partial^2 u}{\partial x_2^2}(x_\lambda) & \cdots & \frac{\partial^2 u}{\partial x_2\partial x_n}(x_\lambda) \\
	\vdots & \vdots & \ddots & \vdots \\
	-\frac{\partial^2 u}{\partial x_n\partial x_1}(x_\lambda) & \frac{\partial^2 u}{\partial x_n\partial x_2}(x_\lambda) & \cdots & \frac{\partial^2 u}{\partial x_n\partial x_n}(x_\lambda)
\end{pmatrix}
=\bar{D}^TD^2u_(x_\lambda)\bar{D},
\end{equation*}
where $\bar{D}=\operatorname{diag}\{-1,1,\cdots,1\}$. Noting that $|\nabla u_\lambda(x)|=|\nabla u(x_\lambda)|$ and the eigenvalue of $D^2u_\lambda(x)$ are the same as $D^2u(x_\lambda)$, especially,
\begin{equation}\label{eq:det}
	\det(D^2u_\lambda(x))=\det(D^2u(x_\lambda)).
\end{equation}

And we define the difference function in direction ${\bf e_1}$, $$U_\lambda(x):=U_{\lambda,{\bf e_1}}(x)=u_\lambda(x)-u(x).$$

Noting that at the special case $x=x_\lambda$, that is, $x\in\overline{T_\lambda}$, we have the following useful results:
\begin{equation}\label{nablaonT}
	\nabla U_\lambda(x)=\left(-2\frac{\partial u}{\partial x_1}(x),0,\dots,0\right);
\end{equation}
\begin{equation}\label{HessainonT}
D^2U_\lambda(x)=
\begin{pmatrix}
0 & -2\frac{\partial^2 u}{\partial x_1\partial x_2}(x) & \cdots & -2\frac{\partial^2 u}{\partial x_1\partial x_n}(x) \\
-2\frac{\partial^2 u}{\partial x_2\partial x_1}(x) & 0 & \cdots & 0 \\
\vdots & \vdots & \ddots & \vdots \\
-2\frac{\partial^2 u}{\partial x_n\partial x_1}(x) & 0 & \cdots & 0
\end{pmatrix}
.
\end{equation}

\vskip 0.2in

Now we are ready to do some preliminary calculations for \eqref{eq:MA}. Firstly, we have \begin{equation}\label{eq:ddet}
	\frac{\partial}{\partial q_{ij}}\det(M)=\det(M)M^{ij}, ~~~\forall M\text{ being positive definite},
\end{equation}
then by the integral form of mean value theorem, we have

 \begin{equation}\label{eq:mv}
 \det\left(D^2u^i_\lambda(x)\right)-\det\left(D^2u^i(x)\right)=\langle {\bf A^i}(x),D^2U^i_\Lambda(x)\rangle_F=\operatorname{tr}\left({\bf A^i}(x)D^2U^i_\lambda(x)\right),
 \end{equation}
 where ${\bf A^i}(x):=(a^i_{jk}(x))_{j,k=1}^n$ with
 \begin{equation}\label{eq:aij}
 a^i_{jk}(x):=\int_0^1\det\left((1-t)D^2u^i_\lambda(x)+tD^2u^i(x)\right)\left((1-t)D^2u^i_\lambda(x)+tD^2u^i(x)\right)^{jk}dt.
 \end{equation}
 
 On the other hand,  $\forall\lambda<0$, $x\in\Sigma_\lambda$, we have $x_1<(x_\lambda)_1<-x_1$, hence by \eqref{eq:det} and \ref{Fanti}, $\forall x\in\Sigma_\lambda$ such that $\frac{\partial u^i}{\partial x_1}(x)\leq0$, we have
 \begin{equation}\label{eq:anti}
 \begin{split}
\det(D^2u^i_\lambda(x))&=\det(D^2u^i(x_\lambda))\\
 &=f^i(x_\lambda,{\bf u}(x_\lambda),\nabla u^i(x_\lambda))\\
 &=f^i(x_\lambda,{\bf u}_\lambda(x),\overline{\nabla u^i_\lambda(x)})\\
 &\geq f^i(x,{\bf u}_\lambda(x),\nabla u^i_\lambda(x)),
 \end{split}
 \end{equation}
 hence by \ref{Fzi},\ref{Fzj} and \ref{Fp}, we have 
 \begin{equation}\label{eq:split}
 \begin{split}
 &\det(D^2u^i_\lambda)-\det(D^2u^i)\\
 \geq& f^i(x,{\bf u}_\lambda,\nabla u^i_\lambda)-f^i(x,{\bf u},\nabla u^i)\\
 =&f^i(x,{\bf u}_\lambda,\nabla u^i_\lambda)-f^i(x,{\bf u},\nabla u^i_\lambda)+f^i(x,{\bf u},\nabla u^i_\lambda)-f^i(x,{\bf u},\nabla u^i)\\
 =&f^i(x,{\bf u}_\lambda,\nabla u^i_\lambda)-f^i(x,u^1_\lambda,\cdots,u^{m-1}_\lambda,u^m,\nabla u^i_\lambda)+\cdots\\
 &+f^i(x,u^1_\lambda,\cdots,u^i_\lambda,\cdots,u^m,\nabla u^i_\lambda)-f^i(x,u^1_\lambda,\cdots,u^i,\cdots,u^m,\nabla u^i_\lambda)+\cdots\\
 &+f^i(x,u^1_\lambda,u^2,\cdots,u^m,\nabla u^i_\lambda)-f^i(x,{\bf u},\nabla u^i_\lambda)\\
 &+f^i(x,{\bf u},\nabla u^i_\lambda)-f^i(x,{\bf u},\nabla u^i)\\
 =&f^i(x,{\bf u}_\lambda,\nabla u^i_\lambda)-f^i(x,u^1_\lambda,\cdots,u^{m-1}_\lambda,u^m,\nabla u^i_\lambda)+\cdots\\
 &+f^{i,1}(x,u^1_\lambda,\cdots,u^i_\lambda,\cdots,u^m,\nabla u^i_\lambda)-f^{i,1}(x,u^1_\lambda,\cdots,u^i,\cdots,u^m,\nabla u^i_\lambda)\\
 &+f^{i,2}(x,u^1_\lambda,\cdots,u^i_\lambda,\cdots,u^m,\nabla u^i_\lambda)-f^{i,2}(x,u^1_\lambda,\cdots,u^i,\cdots,u^m,\nabla u^i_\lambda)+\cdots\\
 &+f^i(x,u^1_\lambda,u^2,\cdots,u^m,\nabla u^i_\lambda)-f^i(x,{\bf u},\nabla u^i_\lambda)\\
 &+f^i(x,{\bf u},\nabla u^i_\lambda)-f^i(x,{\bf u},\nabla u^i)\\
 \geq& d_{im}(x,u^1_\lambda,\cdots,u^{m-1}_\lambda,u^m,\nabla u^i_\lambda,U^m_\lambda)U^m_\lambda+\cdots\\
 &-h_{f^i,z^i}U^i_\lambda+d_{ii}(x,u^1_\lambda,\cdots,u^i,\cdots,u^m,\nabla u^i_\lambda,U^i_\lambda)U^i_\lambda+\cdots\\
 &+d_{i1}(x,{\bf u},\nabla u^i_\lambda,U^1_\lambda)U^1_\lambda-h_{f^i,p}|\nabla U^i_\lambda|,
 \end{split}
 \end{equation}
 where $d_{ij}$ are defined as \eqref{eq:dij}, and $h_{f^i,z^i}$ is the Lipschitz constants of $f^{i,1}$ in \ref{Fzi}, $h_{f^i,p}$ is the Lipschitz constants of $f^i$ in \ref{Fp}.
 
 Next, combining \eqref{eq:mv} and \eqref{eq:split}, we can obtain an elliptic inequality of $U^i_\lambda$ in $\Sigma_\lambda$. $\forall x\in\Sigma_\lambda$ such that $\frac{\partial u^i}{\partial x_1}(x)\leq0$, we have
 \begin{equation}\label{eq:EI}
 \begin{split}
 \operatorname{tr}\left({\bf A^i}(x)D^2U^i_\Lambda(x)\right)+&{\bf B^i}(x)\cdot\nabla U^i_\lambda(x)+c^i(x)U^i_\lambda(x)\\
 \geq&\sum_{j=1}^m d_{ij}(x,u^1_\lambda,\cdots,u^j,\cdots,u^m,\nabla U^i_\lambda,U^j_\lambda)U^j_\lambda,
 \end{split}
 \end{equation}
 with ${\bf A^i}(x)$ defined as \eqref{eq:aij}, and 
\begin{equation}\label{eq:bi}
{\bf B^i(x)}:=\frac{h_{f^i,p}}{|\nabla U^i_\lambda(x))|}\chi_{\{|\nabla U^i_\lambda(x))|\neq0\}}\nabla U^i_\lambda(x),
\end{equation}
\begin{equation}\label{eq:ci}
c^i(x):=h_{f^i,z^i}.
\end{equation}

\vskip 0.2in

Next we give some lemmas here for convenience.
In the procedure of using moving plane methods, the following strong maximum principle and Hopf's Lemma will be crucial. The proof of it can be found in \cite{gilbarg_elliptic_2001}.
\begin{lemma}[Maximum Principle $\&$ Hopf's Lemma]\label{lem:SMP-HL}
Let $\Omega\in\br^n$ be a domain, $w\in C^2(\Omega)$ be a non-positive solution in $\Omega$ to the following elliptic inequality
    \begin{equation*}
    \operatorname{tr}\left(A(x)D^2w(x)\right)+{\bf B}(x)\cdot\nabla w(x)+c(x)w(x)\geq0,
    \end{equation*}
where $A(x):=(a_{ij}(x))_{i,j=1}^n$, ${\bf B}(x):=(b_i(x))_{i=1}^n$, and $a_{ij},b_i,c\in L_{\operatorname{loc}}^\infty(\Omega)$ with $A(x)$ being locally positive definite in $\Omega$. Then either $w\equiv 0$ or $w<0$ in $\Omega$.
	
	Moreover, if $w(x_0)<0$	for some $x_0\in\Omega$, and $w(\bar{x})=0$ for some $\bar{x}\in\partial\Omega$, near which $w$ is continuously differentiable, then	
	$$\frac{\partial w}{\partial \nu}(\bar{x})>0,$$	
	where $\nu$ is the unit outer normal vector of $\partial\Omega$.
\end{lemma}

\vskip 0.2in

In our case, since the domain we dealing with may not satisfy the interior ball condition, we will use the boundary point Hopf lemma at a corner instead, which is the content of the following lemma in \cite{gidas_symmetry_1979} (due to Serrin \cite{serrin_symmetry_1971}).
\begin{lemma}[Serrin's Corner Lemma]\label{lemmaS}
	Let $\Omega$ be a domain in $\br^n$ with the origin $Q$ on its boundary. Assume that near $Q$ the boundary consists of two transversally intersecting $C^2$ hypersurfaces $\{\rho=0\}$ and $\{\sigma=0\}$. Suppose $\rho,\sigma<0$ in $\Omega$. Let $w$ be a negative function in $C^2(\overline{\Omega})$, with $w<0$ in $\Omega$, $w(Q)=0$, satisfying the differential inequality
	\begin{equation*}
	a_{ij}w_{x_ix_j}+b_i(x)w_{x_i}+c(x)w\geq 0\text{ in }\Omega,
	\end{equation*}
	with uniformly bounded coefficients satisfying $a_{ij}\xi_i\xi_j\geq c_0|\xi|^2$. Assume 
	\begin{equation}\label{eq:SL}
	a_{ij}\rho_{x_i}\sigma_{x_j}\geq 0\text{ at }Q.
	\end{equation}
	If this is an equality, assume furthermore that $a_{ij}\in C^2$ in $\overline{\Omega}$ near $Q$, and that 
	\begin{equation*}
	D(a_{ij}\rho_{x_i}\sigma_{x_j})= 0\text{ at }Q,
	\end{equation*}
	for any first order derivative $D$ at $Q$ tangent to the submanifold $\{\rho=0\}\cap\{\sigma=0\}$. Then, for any direction $s$ at $Q$ which enters $\Omega$ transversally to each hypersurface,
	\begin{align*}
	&\frac{\partial w}{\partial s}<0\text{ at $Q$ in case of strict inequality in \eqref{eq:SL},}\\
	&\frac{\partial w}{\partial s}<0\text{ or }\frac{\partial^2 w}{\partial s^2}<0\text{ at $Q$ in case of equality in \eqref{eq:SL}.}
	\end{align*}
\end{lemma}

\vskip 0.2in

\section{Bounded Smooth Simply Connected Domains}\label{sec:bdd}
\noindent

In this case, we mainly consider the following constant-boundary Dirichlet problem for \eqref{eq:MA},
\begin{equation}\label{eq:boundzero}
\left\{\begin{array}{rl}
\det(D^2 u^i,\nabla u^i))\ =&f^i(x,{\bf u},\nabla u^i),\ \text{in}\ \Omega,\\
	u^i\ =&c^i,\quad\ \quad\ \quad\ \quad\ \text{on}\ \partial\Omega,\ 1\leq i\leq m.\\
\end{array}
\right.
\end{equation}
where $c^i\in\br$ are given constants.

\vskip 0.2in

In fact, we can consider a more general case of boundary condition:
\begin{equation}\label{eq:bddbc1}
u^i(y)>u^i(x),\ \forall x\in\Omega, y\in\partial\Omega,\ \text{with}\ y_1<x_1,
\end{equation}
\begin{equation}\label{eq:bddbc2}
u^i(y)\geq u^i(x),\ \forall x,y\in\partial\Omega,\ \text{with}\ y_1<x_1.
\end{equation}
and we deal with the following problem,
\begin{equation}\label{eq:bound}
\left\{\begin{array}{rl}
\det(D^2 u^i,\nabla u^i))\ =&f^i(x,{\bf u},\nabla u^i),\ \text{in}\ \Omega,\\
u^i\ \text{satisfies \eqref{eq:bddbc1} and \eqref{eq:bddbc2}},&1\leq i\leq m.\\
\end{array}
\right.
\end{equation}

\vskip 0.2in

\begin{rmk}\label{rmk:basicineq}
Let ${\bf u}=(u^1,\cdots,u^m)$ be a group of $[C^2(\overline{\Omega})]^m$ solutions of \eqref{eq:bound}, for each $i=1,\dots,m$, fixed $x\in \Omega$, we denote $\rho^i(x)=(\rho^i_1(x),\cdots,\rho^i_n(x))$ as the $n$ eigenvalues of $D^2u^i(x)$, by arithmetic-geometric mean inequality we have
\begin{equation*}
\left(\det D^2u^i(x)\right)^{\frac 1n}=\left(\prod_{j=1}^n\rho^i_j(x)\right)^{\frac 1n}\leq\frac 1n\sum_{j=1}^n\rho^i_j(x)=\frac 1n\Delta u^i(x),
\end{equation*}
then by \ref{F0}, we have
\begin{equation*}
\left\{\begin{array}{rl}
\Delta u^i>0,&\ \text{in}\ \Omega,\\
u^i-\sup\limits_{x\in\partial\Omega}u^i(x)\leq0,&\ \text{on}\ \partial\Omega,\ 1\leq i\leq m.\\
\end{array}
\right.
\end{equation*}
hence by the standard strong maximum principle and Hopf lemma, we have $$u^i<\sup\limits_{x\in\partial\Omega}u^i(x)\ \text{in}\ \Omega,\ \forall\ 1\leq i\leq m,$$
and 
$$\frac{\partial u^i}{\partial \nu}>0\ \text{on}\ \partial\Omega,\ \forall\ 1\leq i\leq m.$$
In particular, we can see the Dirichlet problem \eqref{eq:boundzero} is in the case of \eqref{eq:bound}.
\end{rmk}

\vskip 0.2in

\begin{figure}[ht]
	\centering
	\includegraphics[width=2.5in]{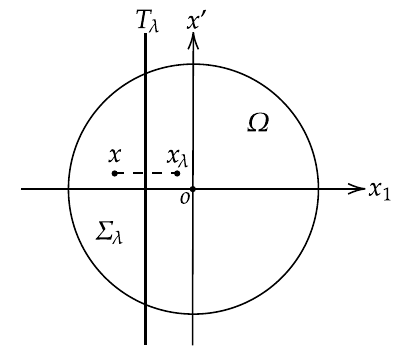}
	\caption{The bounded domain}
	\label{fig:bound}
\end{figure} 

We have some discussion on the geometry of $\Omega$.

When we begin to move the hyperplane $T_\lambda$ along a direction, assumed to be ${\bf e_1}$, from left negative infinity to the left hand side boundary of $\Omega$, denote $\lambda_0:=\inf\limits_{x\in\Omega}x_1$, then the set of firstly touching points of $\{x_1=\lambda\}$ with $\partial\Omega$ is denoted as $\partial_0\Omega:=\{x_1=\lambda_0\}\cap\partial\Omega$. Denote $$\Lambda_0:=\sup\{\lambda>R_0\ |\ \Sigma_\lambda^\lambda\subset\Omega\}.$$ 
$\Lambda_0$ is well-defined due to the regularity of $\Omega$, and when we increase the value of $\lambda$, $\Sigma_\lambda$ will finally reaches at least one of the following two cases:
\begin{enumerate}
	\renewcommand{\theenumi}{(\Roman{enumi})}
	\renewcommand{\labelenumi}{\theenumi}
	\item  $\Sigma_\lambda^\lambda$ becomes internally tangent to $\partial\Omega$ at some point which is not on $\{x_1=\lambda\}$;
	\item $\{x_1=\lambda\}$ reaches a position where it is orthogonal to $\partial\Omega$ at some point.
\end{enumerate}
And we define two more critical value of $\lambda$:
$$\Lambda_1:=\sup\{\lambda>R_0\ |\ \Sigma_\mu \ \text{doesn't reach position (I)},\ \forall \mu\in(R_0,\lambda)\},$$
$$\Lambda_2:=\sup\{\lambda>R_0\ |\ \Sigma_\mu \ \text{doesn't reach position (II)},\ \forall \mu\in(R_0,\lambda)\}.$$
\begin{rmk}
	In general $\Lambda_2\leq\Lambda_0$, while $\Lambda_1$ is irrelevant to $\Lambda_0$.
\end{rmk}

Before stating our main theorem, we prove here the strong maximum principle and Hopf lemma for the fully nonlinear elliptic systems with suitable boundary conditions over $\Sigma_\lambda$, for all $\lambda<\min\{\Lambda_0,\Lambda_1,\Lambda_2\}$:

\begin{lemma}\label{lem:SMPFMA}
	Assume ${\bf f}$ satisfy \ref{F0}, \ref{Fzi}, \ref{Fzj}, \ref{Fp}, \ref{Fanti}, let ${\bf u}$ be $[C^2(\Omega)]^m$ strictly convex solutions to \eqref{eq:bound}, $\lambda<\min\{\Lambda_0,\Lambda_1,\Lambda_2\}$. 
 
 If $$U^i_\lambda\leq0, \frac{\partial u^i}{\partial x_1}\leq0\ \text{in}\ \Sigma_\lambda,\ \text{for all}\ i=1,\dots,m,$$ then $$U^i_\lambda<0\ \text{in}\ \Sigma_\lambda,\ \text{for all}\ i=1,\dots,m,$$ and $$\frac{\partial U^i_\lambda}{\partial x_1}>0,\ \frac{\partial u^i}{\partial x_1}<0\ \text{on}\ T_\lambda\ \text{for all}\ i=1,\dots,m.$$
\end{lemma}
\begin{prf}
    Follows in \eqref{eq:EI}, due to \cref{rmk:dij} and  $U^i_\lambda|_{\Sigma_\lambda}\leq0$, we can obtain an elliptic inequality of $U^i_\lambda$ in $\Sigma_\lambda$: 
    \begin{equation}\label{eq:EIB}
	\operatorname{tr}\left({\bf A^i}(x)D^2U^i_\lambda(x)\right)+{\bf B^i}(x)\cdot\nabla U^i_\lambda(x)+c^i(x)U^i_\lambda(x)\geq0,
    \end{equation}
    where ${\bf A^i}(x)$ as in \eqref{eq:aij} are locally positive definite due to the strictly convexity of $u^i$ and \ref{F0}, and together with ${\bf B^i(x)}$, $c^i(x)$ as in \eqref{eq:bi}, \eqref{eq:ci} are all locally bounded due to the twice differentiable continuity of $u^i$. Hence by \cref{lem:SMP-HL}, we have either $U^i_\lambda<0$ or $U^i_\lambda\equiv0$ in $\Sigma_\lambda$, it's easy to see that the latter will not happen due to $U^i_\lambda|_{\partial\Sigma_\lambda\setminus\overline{T_\lambda}}<0$ by the boundary condition \eqref{eq:bddbc1} and $\lambda<\min\{\Lambda_0,\Lambda_1,\Lambda_2\}$. Now noticing that $U^i_\lambda|_{T_\lambda}\equiv0$, by \cref{lem:SMP-HL} again we have $\frac{\partial U^i_\lambda}{\partial x_1}>0$ on $T_\lambda$, and hence as in \eqref{nablaonT}, $\frac{\partial u^i}{\partial x_1}=-\frac12\frac{\partial U^i_\lambda}{\partial x_1}<0$ on $T_\lambda$.
\end{prf}

\vskip 0.2in
\subsection{Main Theorem}
We now state our main theorem for case of bounded smooth simply connected domains.

When $\Omega$ being convex in one direction, assumed as $\bf e_1$, and the nonlinear term $\bf f$ satisfies some corresponding monotonic conditions in this direction, we can start to examine whether the solution of the system \eqref {eq:bound} will satisfy the corresponding monotonicity along this direction. The main results are the following.

\begin{thm}\label{thm:bound}
	Let $\Omega\subset\br^n$ be a $C^2$ bounded simply connected domain, convex in ${\bf e_1}$ direction. Assume ${\bf f}$ satisfy \ref{Fc}, \ref{Fzi}, \ref{Fzj}, \ref{Fp}, \ref{Fanti}. Let ${\bf u}=\left(u^1,\cdots,u^m\right)$ be a group of $\left[C^2(\overline{\Omega})\right]^m$ strictly convex solutions to \eqref{eq:bound}, then each $u^i$ must be
	$$u^i\left(x_1,x'\right)\geq u^i\left(2\Lambda_0-x_1,x'\right)\ \text{and}\ \frac{\partial u^i}{\partial x_1}(x)<0\ \text{in}\ \{x\in \Omega\ |\ x_1<\Lambda_0\}.$$
	
    Furthermore, if $\frac{\partial u^i}{\partial x_1}\left(\Lambda_0,x'\right)=0$ for some $x\in \left\{x_1=\Lambda_0\right\}$, then such $u^i$ must be symmetric with respect to $\left\{x_1=\Lambda_0\right\}$ and strictly decreasing in ${\bf e_1}$ direction with $x_1<\Lambda_0$, more precisely, 
	\begin{equation*}
	u^i(x)=u^i\left(\left|x_1-\Lambda_0\right|,x'\right)\ \text{in}\ \left\{x\in \Omega\ |\ |x_1-\Lambda_0|<\Lambda_0-\lambda_0\right\},
	\end{equation*}
	moreover, 
	\begin{equation*}
	\frac{\partial u^i}{\partial x_1}(x)<0\ \text{in}\ \left\{x\in \Omega\ |\ x_1<\Lambda_0\right\}.
	\end{equation*}
\end{thm}

\vskip 0.2in

\begin{rmk}
    As we can see in the proof, if $\Omega$ satisfy $\Lambda_2=\Lambda_0$, then the situation (II) will not be happened. Hence we can weaken then condition \ref{Fc} to be \ref{F0}, i.e. we don't need $f^i$ to be strictly positive.
\end{rmk}

\vskip 0.2in

If we assume more symmetry condition on $\Omega$ and $f^i$ (substituting \ref{Fanti} with \ref{Fsym1}), we can furthermore immediately have the following, by using \cref{thm:bound} again with ${\bf u_{\Lambda_0}}:=(u^i_{\Lambda_0})$. (Noting that in this case, the inequalities \eqref{eq:anti},\eqref{eq:split} will be slightly different to obtain the same result.)

\begin{thm}
	Let $\Omega\subset\br^n$ be a $C^2$ bounded simply connected domain, convex in ${\bf e_1}$ direction, and symmetric with respect to $\{x_1=\Lambda_0\}$. Assume ${\bf f}$ satisfy \ref{Fc}, \ref{Fzi}, \ref{Fzj}, \ref{Fp}, \ref{Fsym1}. Let ${\bf u}=(u^1,\cdots,u^m)$ be a group of $[C^2(\overline{\Omega})]^m$ strictly convex solutions of \eqref{eq:bound}, then each $u^i$ must be	symmetric with respect to $\{x_1=\Lambda_0\}$ and strictly decreasing in ${\bf e_1}$ direction with $x_1<\Lambda_0$. 
	
	More precisely, for all $i=1,\dots, m$,
	\begin{equation*}
	u^i(x)=u^i\left(\left|x_1-\Lambda_0\right|,x'\right),\ \text{in}\ \mleft\{x\in \Omega\ \middle|\ \left|x_1-\Lambda_0\right|<\Lambda_0-\lambda_0\mright\},
	\end{equation*}
	moreover, 
	\begin{equation*}
	\frac{\partial u^i}{\partial x_1}(x)<0\ \text{in}\ \{x\in \Omega\ |\ x_1<\Lambda_0\}.
	\end{equation*}
\end{thm}

\vskip 0.2in

Especially, when $\Omega$ is a ball, if we substitute \ref{Fanti} to the symmetric one \ref{Fsymall}, then by using \cref{thm:bound} respect to all directions in $\br^n$, we have immediately that

\begin{coro}\label{thm:ball}
	Let $\Omega=B_R$ be a arbitrary ball with radius $R$. Assume ${\bf f}$ satisfy \ref{F0}, \ref{Fzi}, \ref{Fzj}, \ref{Fp}, \ref{Fsymall}. Let ${\bf u}=(u^1,\cdots,u^m)$ be a group of $[C^2(\overline{\Omega})]^m$ strictly convex solutions to \eqref{eq:bound}, then each $u^i$ must be radially symmetric and strictly increasing respect to the center of $B_R$. 
	
	More precisely, denote the center of $B_R$ as $x^*\in\br^n$, and $r=|x-x^*|$, then for $i=1\dots,m$, each $u^i$ must be
	\begin{equation*}
	u^i(x)=u^i(r),\ \forall\ x\in B_R(x^*),
	\end{equation*}
	moreover,
	\begin{equation*}
	\frac{d u^i}{dr}(x)>0,\ \forall\ x\in B_R(x^*),
	\end{equation*}
\end{coro}

\vskip 0.2in
\subsection{Proof of \cref{thm:bound}}

We are now in a position to prove the theorem.

\newenvironment{prf1}{{\noindent\bf Proof of \cref{thm:bound}.}}{\hfill $\square$\par}
\begin{prf1}.
	
{\bf Step 1}: There exists a real number $\lambda<\min\{\Lambda_0,\Lambda_1,\Lambda_2\}$ such that $\forall\mu<\lambda$, $\forall 1\leq i\leq m$, $\forall x\in\Sigma_\mu$, $U^i_\mu(x)<0$.

Recall that $\partial_0\Omega$ is the set of firstly touching point of $\partial\Omega$ with $\{x_1=\lambda_0\}$. Noting that by \cref{rmk:basicineq}, we have $\frac{\partial u^i}{\partial x_1}(x)<0, \forall x\in\partial_0\Omega$, hence by the continuity of $\nabla u^i$ up to the boundary, for sufficiently small $\epsilon_x>0$, we must have $$\frac{\partial u^i}{\partial x_1}<0\ \text{on}\ \Omega\cap B_{\epsilon_x}(x), \ \forall\ 1\leq i\leq m.$$ Noting that $\partial_0\Omega$ is a compact set, hence we must have finite cover of it by $\{B_{\epsilon_x}(x)\}$, denoted as $\{B_{\epsilon_{x_i}}(x_i)\}_{i=1}^K$. 

Now denote $$\Omega_\epsilon:=\Omega\cap\bigcup\limits_{i=1}^K B_{\epsilon_{x_i}}(x_i),$$ we have $$\frac{\partial u^i}{\partial x_1}(x)<0,\ x\in \Omega_\epsilon,\ \forall\ i=1,\dots,m, $$ hence when $\lambda$ sufficiently close to $\lambda_0$ (such that $\Sigma_\lambda\cup\Sigma_\lambda^\lambda\subset\Omega_\epsilon$), $\forall \mu\in(R_0,\lambda)$, we have $$U^i_\mu(x)=\int_{x_1}^{2\mu-x_1}\frac{\partial u^i}{\partial x_1}(s,x')ds<0,\ x\in \Sigma_\mu,$$ and we have done.

{\bf Step 2}: Let $\Lambda:=\sup\{\lambda<\Lambda_0\ |\ U^i_\mu(x)<0,\ \forall x\in\Sigma_\mu,\ \forall\ 1\leq i\leq m,\ \forall\ \mu<\lambda\}$, we want to prove that $\Lambda=\Lambda_0$.

If not, that is $\Lambda<\Lambda_0$, then $U^i_\Lambda\leq 0$ by the continuity of $u^i$, further more, $U^i_\mu\leq 0, \forall \mu\leq\Lambda$, it's clear to see by covering argument that $\frac{\partial u^i}{\partial x_1}(x)\leq0, \forall x\in\Sigma_\Lambda$. Hence by \cref{lem:SMPFMA}, $U^i_\Lambda<0$ in $\Sigma_\Lambda$ and $\frac{\partial U^i_\Lambda}{\partial x_1}>0$, $\frac{\partial u^i}{\partial x_1}<0$ on $T_\Lambda$ for all $i=1,\dots, m$.

We consider a sequence of $\lambda_k\in(\Lambda,\Lambda_0)$ converging to $\Lambda$ and a sequence of $x_k\in\Sigma_{\lambda_k}$ such that $U^i_{\lambda_k}(x_k)\geq 0$ for a specific $i$ (at least one of them verifies this for infinitely many $k$'s). Since $U^i_{\lambda_k}|_{\partial\Sigma_{\lambda_k}}\leq0$ by boundary conditions, we can substitute $\{x_k\}$ to be lying in the interior of $\Sigma_{\lambda_k}$ such that $$U^i_{\lambda_k}(x_k)=\max\limits_{\overline{\Sigma_{\lambda_k}}}U^i_{\lambda_k}\geq0,$$ and hence $\nabla U^i_{\lambda_k}(x_k)=0.$
By passing to a subsequence if necessary, due to the boundedness of $\Omega$, we have $x_k\to x^*\in\overline{\Sigma_\Lambda}$, and
\begin{equation}\label{eq:bddmax} U^i_{\Lambda}(x^*)\geq 0,
\end{equation} 
\begin{equation}\label{eq:bddgrad} 
\nabla U^i_{\Lambda}(x^*)=0. 
\end{equation} 

Noting that $U^i_\Lambda|_{\Sigma_\Lambda}<0$ shows that $x^*\in\partial\Sigma_\Lambda=\partial\Omega\cup T_\Lambda$, while $\left.\frac{\partial U_\Lambda^i}{\partial x_1}\right|_{T_\Lambda}>0$ shows that $x^*\in\partial\Omega$. Due to the geometry of $\Omega$, there are still two cases that could occur:
\begin{enumerate}
	\renewcommand{\theenumi}{\textbf{Case \arabic{enumi}}}
	\renewcommand{\labelenumi}{\theenumi}
	
	\item $x^*\in\partial\Omega\setminus\overline{T_\Lambda}$. 
	
	In this case, boundary conditions \eqref{eq:bddbc1} and $U^i_{\Lambda}(x^*)\geq 0$ shows that $x^*\in\partial\Sigma_\Lambda\setminus\{x\in\partial\Omega\ |\ x^\Lambda\in\Omega\}$, then we have $\left((x^*)^\Lambda\right)\in\partial\Omega$. In fact, $x^*$ is on the position where the situation (I) occurs. Hence the unit interior normal vectors of these two points must be coincided with each other and both be orthogonal to ${\bf e_1}$, without lost of generality, we denote them as ${\bf e_n}$. 
	
	Since $U^i_{\Lambda}(x^*)\geq 0$ and boundary condition \eqref{eq:bddbc2} shows that $U^i_\Lambda(x^*)=0$, noting that \eqref{eq:EIB} also holds on $\Sigma_\Lambda$, then \cref{lem:SMP-HL} ensures that 
	\begin{equation}\label{eq:bdd1}
	\frac{\partial U^i_\Lambda}{\partial x_n}(x^*)<0,
	\end{equation}
	which is contradictory to \eqref{eq:bddgrad}.
	
	\item $x^*\in\partial\Omega\cap\overline{T_\Lambda}$.
	
	In this case, denote the unit outer normal vector of $x^*$ as $\nu$, 

	suppose that $\nu_1<0$, then $U^i_\Lambda(x^*)=0$. Using \cref{lem:SMPFMA} on \eqref{eq:EIB}, which holds over $\Sigma_\Lambda$ in this case, shows that $\frac{\partial U^i_\Lambda}{\partial x_1}(x^*)>0$, which leads a contradiction to \eqref{eq:bddgrad}. Hence $\nu_1=0$, which means that $x^*$ is on the position where the situation (II) happens.
	
	Without lost of generality, we assume $\nu=-{\bf e_n}$. Choosing $s:=-{\bf e_1}+{\bf e_n}$ as the non-tangent direction entering $\Sigma_\Lambda$, noting that \eqref{eq:EIB} still holds over $\Sigma_\Lambda$ now, and \eqref{eq:det} together with \eqref{eq:aij} shows that $a_{1j}=a_{j1}=0,\forall j=2,\dots,n$. Now since $T_\Lambda$ is tangent with $\partial\Omega$ at $x^*$, it can be locally regarded as $\left\{\rho\equiv x_1-\Lambda=0 \right\}$ intersect with $\left\{\sigma\equiv x_2+\dots+x_n=0\right\}$ at $x^*$, hence $a_{ij}\rho_i\sigma_j=0$ at $x^*$, \cref{lemmaS} shows that either \begin{equation}\label{eq:bdd4}
	\frac{\partial U^i_\Lambda}{\partial s}(x^*)=-\frac{\partial U^i_\Lambda}{\partial x_1}(x^*)+\frac{\partial U^i_\Lambda}{\partial x_n}(x^*)<0,
	\end{equation}
	or 
	\begin{equation}\label{eq:bdd5}
	\frac{\partial^2 U^i_\Lambda}{\partial s^2}(x^*)=\frac{\partial^2 U^i_\Lambda}{\partial x_1^2}(x^*)-2\frac{\partial^2 U^i_\Lambda}{\partial x_1\partial x_n}(x^*)+\frac{\partial^2 U^i_\Lambda}{\partial x_n^2}(x^*)<0. 
	\end{equation}
	
	Noting that \eqref{nablaonT} and \eqref{HessainonT}, together with \eqref{eq:bddgrad}, \eqref{eq:bdd4} and \eqref{eq:bdd5}, shows that $\frac{\partial^2 U^i_\Lambda}{\partial x_1\partial x_n}(x^*)>0$.
	
	Consider the segment $I_k$ in the ${\bf -e_1}$ directing from $x_k$ to $y_k\in T_\Lambda$. Then, due to continuity, for sufficiently large $k$ such that $\frac{\partial^2 U^i_\Lambda}{\partial x_1\partial x_n}>0$ holds in $I_k$, we have 
	$$ 0=\frac{\partial U^i_\Lambda}{\partial x_n}(x_k)=	\frac{\partial U^i_\Lambda}{\partial x_n}(y_k)+\int_{y_{k,1}}^{x_{k,1}}\frac{\partial^2 U^i_\Lambda}{\partial x_1\partial x_n}(s,x')ds>0$$
	which is a contradiction.
\end{enumerate}

In summary, the above two cases would not happen, and hence $\Lambda=\Lambda_0$.

{\bf Step 3}: Conclusions.

Now $\Lambda=\Lambda_0$, then $U^i_\Lambda\leq 0$ by the continuity of $u^i$, and $\frac{\partial u^i}{\partial x_1}<0$ in $\{x\in \Omega\ |\ x_1<\Lambda_0\}$ by \cref{lem:SMPFMA} on $\Sigma_\lambda,\ \forall\lambda<\Lambda$ with a covering argument.

For the second part assertion in the theorem, noting that \eqref{eq:EIB} also holds on $\Sigma_\Lambda$, hence $\frac{\partial u^i}{\partial x_1}(\Lambda_0,x')=0$ for some $x\in \{x_1=\Lambda_0\}$ and \cref{lem:SMPFMA} shows that $U^i_\Lambda\equiv0$ in $\Sigma_\Lambda$, and then we finish the proof of the whole theorem.
\end{prf1}

\vskip 0.2in

\section{Bounded Tube Shape Domains}
\label{sec:bt}
\noindent

Now we turn our attention on the bounded tubes, any tubes in $\br^n$ can always regarded as $C_H:=\Omega\times(-H,H)$ up to rotations and translations, where $\Omega\subset\br^{n-1}$ is a bounded simply connected domain, $H>0$, we assume $\partial\Omega\in C^2$, and denote the upper surface as $C_u:=\Omega\times\{x_n=H\}$, and the lower surface as $C_l:=\Omega\times\{x_n=-H\}$. 

In this case, we mainly consider the following constant-boundary Dirichlet problem for \eqref{eq:MA}
\begin{equation}\label{eq:cylinder}
\left\{\begin{array}{rl}
\det(D^2 u^i)\ =&f^i(x,{\bf u},\nabla u^i),\ x\in\ C_H,\\
u^i\ =&c^i,\hspace{1.9cm}x\in\partial C_H,\  i=1,\dots,m.\\
\end{array}
\right.
\end{equation}
where $c^i\in\br$ are constants.

\vskip 0.2in

Similar to the case of bounded smooth simply connected domain case, when we move the hyperplane $T_\lambda$ along a direction, assumed to be ${\bf e_1}$, from left negative infinity to the left hand side boundary of $\Omega$, denote $\lambda_0:=\inf\mleft\{x_1\ \middle|\ x\in C_H\mright\}$, then the set of firstly touching points of $\left\{x_1=\lambda\right\}$ with $\partial C_H$ is denoted as $\partial_0 C_H:=\{x_1=\lambda_0\}\cap \partial C_H$. And as we continue to move, we can define $$\Lambda_0:=\sup\mleft\{\lambda>\lambda_0\ \middle|\  \Sigma_\lambda^\lambda\subset\Omega\mright\},$$ similarly, with two probably happened situations, and define $\Lambda_1,\Lambda_2$ respectively. 

\vskip 0.2in
\subsection{Main Theorem}
We now state our main theorem for case of bounded tubes.

When $\Omega$ being convex in one direction, assumed as $\bf e_1$, and the nonlinear term $\bf f$ satisfies some corresponding monotonic conditions in this direction, we can start to examine whether the solution of the system \eqref {eq:cylinder} will satisfy the corresponding monotonicity along this direction. The main results are the following.

\begin{thm}\label{thm:cylinder}
	Let $C_H=\Omega\times(-H, H)$ be a $C^2$ bounded tubes in $\br^n$, where $\Omega\subset\br^{n-1}$ is a $C^2$ bounded simply connected domain being convex along with ${\bf e_1}$. Assume ${\bf f}$ satisfy \ref{Fc}, \ref{Fzi}, \ref{Fzj}, \ref{Fp}, \ref{Fanti}. Let ${\bf u}=\left(u^1,\cdots,u^m\right)\in \left[C^2\left(\overline{C_H}\right)\right]^m$  be a group of strictly convex solutions to \eqref{eq:cylinder}, then each $u^i$ must be
	$$u^i\left(x_1,x'\right)\geq u^i\left(2\Lambda_0-x_1,x'\right)\ \text{and}\ \frac{\partial u^i}{\partial x_1}(x)<0\ \text{in} \mleft\{x\in C_H\ \middle|\ x_1<\Lambda_0\mright\}.$$
	
	Furthermore, if
	\begin{equation}\label{eq:nd=0cylinder}
	\frac{\partial u^i}{\partial x_1}\left(\Lambda_0,x'\right)=0 ~for~some~ x\in \left\{x_1=\Lambda_0\right\},
	\end{equation} then such $u^i$ must be symmetric with respect to $\left\{x_1=\Lambda_0\right\}$ and strictly decreasing in ${\bf e_1}$ direction with $x_1<\Lambda_0$, more precisely, 
	\begin{equation*}
	u^i(x)=u^i\left(\left|x_1-\Lambda_0\right|,x'\right),\ \text{in}\ \mleft\{x\in C_H\ \middle|\ \left|x_1-\Lambda_0\right|<\Lambda_0-\lambda_0\mright\},
	\end{equation*}
	moreover, 
	\begin{equation*}
	\frac{\partial u^i}{\partial x_1}(x)<0,\ \text{in}\ \mleft\{x\in C_H\ \middle|\ x_1<\Lambda_0\mright\}.
	\end{equation*}
\end{thm}

\vskip 0.2in

If we assume more symmetry condition on $\Omega$ and ${\bf f}$ (substituting \ref{Fanti} with \ref{Fsym1}) to satisfy \eqref{eq:nd=0cylinder}, we can furthermore immediately have the following, by using \cref{thm:cylinder} again with ${\bf u_{\Lambda_0}}:=(u^i_{\Lambda_0})$. (Noting that in this case, the inequalities \eqref{eq:anti},\eqref{eq:split} will be slightly different to obtain the same result.)

\begin{thm}
    Let $C_H=\Omega\times(-H, H)$ be a $C^2$ bounded tubes in $\br^n$, where $\Omega\subset\br^{n-1}$ is a $C^2$ bounded simply connected domain being convex along with ${\bf e_1}$, and symmetric with respect to $\{x_1=\Lambda_0\}$. Assume ${\bf f}$ satisfy \ref{Fc}, \ref{Fzi}, \ref{Fzj}, \ref{Fp}, \ref{Fsym1}. Let ${\bf u}=(u^1,\cdots,u^m)$ be a group of $\left[C^2(\overline{C_H})\right]^m$ strictly convex solutions to \eqref{eq:cylinder}, then each $u^i$ must be symmetric with respect to $\left\{x_1=\Lambda_0\right\}$, and strictly decreasing in $\left\{x_1<\Lambda_0\right\}$ in ${\bf e_1}$ direction with $x_1<\Lambda_0$. 
	
	More precisely, for all $i=1,\dots, m$,
	\begin{equation*}
	u^i(x)=u^i\left(\left|x_1-\Lambda_0\right|,x'\right),\ \text{in}\ \mleft\{x\in C_H\ \middle|\ \left|x_1-\Lambda_0\right|<\Lambda_0-\lambda_0\mright\},
	\end{equation*}
	moreover, 
	\begin{equation*}
	\frac{\partial u^i}{\partial x_1}(x)<0,\ \text{in}\ \mleft\{x\in C_H\ \middle|\ x_1<\Lambda_0\mright\}.
	\end{equation*}
\end{thm}

\vskip 0.2in

Especially, when $\Omega$ is a ball, that is $C_H$ being a cylinder, if we substitute \ref{Fanti} to the symmetric ones \ref{Fsymall}, then by using \cref{thm:cylinder} respect to all directions in $\br^n$, we have immediately that

\begin{coro}\label{thm:cylinderball}
	Let $C_H=\Omega\times(-H, H)$ be a cylinder in $\br^n$, that is $\Omega=B_R$ be a arbitrary ball with radius $R$ in $\br^{n-1}$. Assume ${\bf f}$ satisfy \ref{Fc}, \ref{Fzi}, \ref{Fzj}, \ref{Fp}, \ref{Fsymall}. Let ${\bf u}=(u^1,\cdots,u^m)$ be a group of $[C^2(\overline{C_H})]^m$ strictly convex solutions to \eqref{eq:cylinder}, then each $u^i$ must be radially symmetric and strictly increasing respect to the axis crossing the center of $B_R$.
	
	More precisely, denote the center of $B_R$ as $x^*=\left((x^*)',x^*_n\right)\in\br^n$, and denote $r=|x'-(x^*)'|$, then for $i=1\dots,m$, each $u^i$ must be
	\begin{equation*}
	u^i(x)=u^i(r,x_n),\ x\in C_H,
	\end{equation*}
	moreover, 
	\begin{equation*}
	\frac{\partial u^i}{\partial r}(r,x_n)>0,\ x\in C_H.
	\end{equation*}
\end{coro}

\vskip 0.2in
\subsection{Proof of \cref{thm:cylinder}}

We begin to prove the theorem.

\newenvironment{proof3}{{\noindent\bf Proof of \cref{thm:cylinder}.}}{\hfill $\square$\par}
\begin{proof3}
	
	{\bf Step 1}: There exists a real number $\lambda<\min\{\Lambda_0,\Lambda_1,\Lambda_2\}$ such that $\forall\ \mu<\lambda$, $\forall i=1,\dots,m$, $\forall x\in\Sigma_\mu$, $U^i_\mu(x)<0$.
	
    For $i=1,\dots,m$, noticing that  \cref{rmk:basicineq} still holds for $u^i$ here. We firstly focus on $x_H:=(\lambda_0,x_0'',H)$. For $s:={\bf e_1}-{\bf e_n}$, at $x_H$, we can locally regarded as two planes $\mleft\{\rho\equiv x_1-\lambda_0=0 \mright\}$ and $\mleft\{\sigma\equiv x_n-H=0\mright\}$ intersecting with each other, hence $a_{ij}\rho_i\sigma_j=0$, by \cref{lemmaS}, we have either 
	\begin{equation}\label{eq:cyl1}
	\frac{\partial u^i}{\partial s}(x_H)=\frac{\partial u^i}{\partial x_1}(x_H)-\frac{\partial u^i}{\partial x_n}(x_H)<0,
	\end{equation}
	or
	\begin{equation}\label{eq:cyl2}
	\frac{\partial^2 u^i}{\partial s^2}(x_H)=\frac{\partial^2 u^i}{\partial x_1^2}(x_H)-2\frac{\partial^2 u^i}{\partial x_1\partial x_n}(x_H)+\frac{\partial^2 u^i}{\partial x_n^2}(x_H)<0. 
	\end{equation}
	However at $x_H$, by boundary conditions, we have $$\frac{\partial u^i}{\partial x_1}(x_H)=0, \frac{\partial^2 u^i}{\partial x_1^2}(x_H)=0,\frac{\partial u^i}{\partial x_n}(x_H)=0, \frac{\partial^2 u^i}{\partial x_n^2}(x_H)=0,$$ which showing that \eqref{eq:cyl1} not hold, hence by \eqref{eq:cyl2},
	\begin{equation}\label{eq:cyl3}
		\frac{\partial^2 u^i}{\partial x_1\partial x_n}(x_H)>0.
	\end{equation}
	Now, by the $C^2$ continuity of $u^i$ up to the boundary, we can choose $\epsilon_1>0$, such that for any $x\in C_H, |x-x_H|<\epsilon_1$ close to $x_H$, we have $\frac{\partial^2 u^i}{\partial x_1\partial x_n}(x)>0$. Denote $x_u$ be the point of $x$ directing on the upper surface $C_u$, then by \eqref{eq:cyl3},
	$$\frac{\partial u^i}{\partial x_1}(x_u)-\frac{\partial u^i}{\partial x_1}(x)=\int_{x_n}^{(x_u)_n}\frac{\partial^2 u^i}{\partial x_1\partial x_n}(x',s)ds>0,$$
	noticing that $\frac{\partial u^i}{\partial x_1}(x_u)=0$, we have $\frac{\partial u^i}{\partial x_1}(x)<0$.
	Similarly, for $x_{-H}:=(\lambda_0,x_0'',-H)$, consider $s={\bf e_1}+{\bf e_n}$ with the same argument, there exists $\epsilon_2>0$, such that for any $x\in C_H, |x-x_{-H}|<\epsilon_2$, we have $\frac{\partial u^i}{\partial x_1}(x)<0$.
	
	Next, we consider the domain $C_H':=C_H\setminus \left(B_{\epsilon_1}(x_H)\cup B_{\epsilon_2}(x_{-H})\right)$. For any $x\in\partial C_H'\cap\{x_1=\lambda_0\}$, we have $\frac{\partial u^i}{\partial x_1}(x)<0$ by \cref{rmk:basicineq}, hence there exists $\epsilon_x>0$, such that for any $y\in C_H, |x-y|<\epsilon_x$, we have $\frac{\partial u^i}{\partial x_1}(y)<0$. Noticing that $C_H'\cap\{x_1=\lambda_0\}$ is compact, there must be a finite cover by $\{B_{\epsilon_x}(x)\}$, denote as$\{B_{\epsilon_{x_i}}(x_i)\}_{i=1}^K$.
	
	Denote $$C_\epsilon:= C_H\cap \left(B_{\epsilon_1}(x_H)\cup B_{\epsilon_2}(x_{-H})\bigcup\limits_{i=1}^K B_{\epsilon_{x_i}}(x_i)\right),$$ we have $$\frac{\partial u^i}{\partial x_1}(x)<0,\ x\in C_\epsilon,\ \forall i=1,\dots,m,$$ hence when $\lambda$ sufficiently close to $-\lambda_0$, such that $\Sigma_\lambda\cup\Sigma_\lambda^\lambda\subset C_\epsilon$, then for any $\mu\in(-R,\lambda)$, we have $$U^i_\mu(x)=\int_{x_1}^{2\mu-x_1}\frac{\partial u^i}{\partial x_1}(s,x')ds<0,\ x\in \Sigma_\mu,$$ which accomplishes our first step.
	
	{\bf Step 2}: Let $\Lambda:=\sup\{\lambda<\Lambda_0\ |\ U^i_\mu(x)<0,\ \forall x\in\Sigma_\mu,\ \forall\ i=1,\dots,m,\ \forall\ \mu<\lambda\}$, we need to prove $\Lambda=\Lambda_0$.
	
	If not, that is $\Lambda<\Lambda_0$, then $U^i_\Lambda\leq 0$ by the continuity of $u^i$, further more, $U^i_\mu\leq 0, \forall \mu\leq\Lambda$, it's clear to see by covering argument that $\frac{\partial u^i}{\partial x_1}(x)\leq0, \forall x\in\Sigma_\Lambda$. Noting that \cref{lem:SMPFMA} still holds in $C$, hence for all $i=1,\dots, m$, $U^i_\Lambda<0$ in $\Sigma_\Lambda$ and $\frac{\partial U^i_\Lambda}{\partial x_1}>0$, $\frac{\partial u^i}{\partial x_1}<0$ on $T_\Lambda$.
	
	We consider a sequence of $\lambda_k\in(\Lambda,\Lambda_0)$ converging to $\Lambda$ and a sequence of  $x_k\in\Sigma_{\lambda_k}$ such that $U^i_{\lambda_k}(x_k)\geq 0$ for a specific $i$ (at least one of them verifies this for infinitely many $k$'s). By boundary conditions and \cref{rmk:basicineq}, we have $\left.U^i_{\lambda_k}\right|_{\partial\Sigma_{\lambda_k}}\leq0$, hence we can substitute $\{x_k\}$ to be lying in the interior of $\Sigma_{\lambda_k}$, such that $$U^i_{\lambda_k}(x_k)=\max\limits_{\overline{\Sigma_{\lambda_k}}}U^i_{\lambda_k}\geq0,$$ and hence $\nabla U^i_{\lambda_k}(x_k)=0.$ By passing to a subsequence if necessary, due to the boundedness of $C_H$, we have $x_k\to x^*\in\overline{\Sigma_\Lambda}$, and
	\begin{equation}\label{eq:cylmax} U^i_{\Lambda}(x^*)\geq 0,
	\end{equation} 
	\begin{equation}\label{eq:cylgrad} 
	\nabla U^i_{\Lambda}(x^*)=0. 
	\end{equation} 
	
	Noting that $\left.U^i_\Lambda\right|_{\Sigma_\Lambda}<0$ shows that  $x^*\in\partial\Sigma_\Lambda\subset\partial C_H\cup T_\Lambda$, while $\left.\frac{\partial U_\Lambda^i}{\partial x_1}\right|_{T_\Lambda}>0$ shows that $x^*\in\partial C_H$. Due to the geometry of $C_H$, there are still two cases that could occur:
	
	\begin{enumerate}
	\renewcommand{\theenumi}{\textbf{Case \arabic{enumi}}}
	\renewcommand{\labelenumi}{\theenumi}
 
    \item 
    $x^*\in\partial C_H\setminus\overline{T_\Lambda}$. 
		
	In this case, boundary conditions and $U^i_{\Lambda}(x^*)\geq 0$ shows that $x^*\in\partial\Sigma_\Lambda\setminus\mleft\{x\in\partial C_H\ \middle|\ x^\Lambda\in C_H\mright\}$, then we have $\left((x^*)^\Lambda\right)\in\partial C_H$. In fact, $x^*$ is on the position where the situation (I) occurs. Hence the unit interior normal vectors of these two points must be coincided with each other and both be orthogonal to  ${\bf e_1}$, without lost of generality, we denote them as ${\bf e_n}$. 
		
	Boundary condition shows that $U^i_\Lambda(x^*)=0$, noting that \eqref{eq:EIB} also holds on $\Sigma_\Lambda$, then \cref{lem:SMP-HL} ensures that
	\begin{equation*}
		\frac{\partial U^i_\Lambda}{\partial x_n}(x^*)>0,
	\end{equation*}
	which is contradictory to \eqref{eq:cylgrad}.
		
	\item
    $x^*\in\partial C_H\cap\overline{T_\Lambda}$.
		
	In this case, denote the unit outer normal vector of $x^*$ as $\nu$, suppose that $\nu_1<0$, then $U^i_\Lambda(x^*)=0$. Using \cref{lem:SMPFMA} on \eqref{eq:EIB}, which holds over $\Sigma_\Lambda$ in this case, shows that $\frac{\partial U^i_\Lambda}{\partial x_1}(x^*)>0$, which leads a contradiction to \eqref{eq:bddgrad}. Hence $\nu_1=0$, which means that $x^*$ is on the position where the situation (II) happens.
		
	Without lost of generality, we assume $\nu=-{\bf e_n}$. Choosing $s:=-{\bf e_1}+{\bf e_n}$ as the non-tangent direction entering $\Sigma_\Lambda$, noting that \eqref{eq:EIB} still holds over $\Sigma_\Lambda$ now, and \eqref{eq:det} together with \eqref{eq:aij} shows that $a_{1j}=a_{j1}=0,\forall j=2,\dots,n$. Now since $T_\Lambda$ is tangent with $\partial C_H$ at $x^*$,  it can be locally regarded as $\mleft\{\rho\equiv x_1-\Lambda=0 \mright\}$ intersect with $\mleft\{\sigma\equiv x_2+\dots+x_n=0\mright\}$ at $x^*$, hence $a_{ij}\rho_i\sigma_j=0$ at $x^*$, \cref{lemmaS} shows that either
	\begin{equation}\label{eq:cyl4}
		\frac{\partial U^i_\Lambda}{\partial s}(x^*)=-\frac{\partial U^i_\Lambda}{\partial x_1}(x^*)+\frac{\partial U^i_\Lambda}{\partial x_n}(x^*)<0,
	\end{equation}
	or
	\begin{equation}\label{eq:cyl5}
		\frac{\partial^2 U^i_\Lambda}{\partial s^2}(x^*)=\frac{\partial^2 U^i_\Lambda}{\partial x_1^2}(x^*)-2\frac{\partial^2 U^i_\Lambda}{\partial x_1\partial x_n}(x^*)+\frac{\partial^2 U^i_\Lambda}{\partial x_n^2}(x^*)<0. 
	\end{equation}
		
	Noting that \eqref{nablaonT} and \eqref{HessainonT}, together with \eqref{eq:cylgrad}, \eqref{eq:cyl4} and \eqref{eq:cyl5}, shows that $\frac{\partial^2 U^i_\Lambda}{\partial x_1\partial x_n}(x^*)>0$.
		
	Consider the segment $I_k$ in the ${\bf -e_1}$ directing from $x_k$ to $y_k\in T_\Lambda$. Then, due to continuity, for sufficiently large $k$ such that $\frac{\partial^2 U^i_\Lambda}{\partial x_1\partial x_n}>0$ holds in $I_k$, we have 
	$$ 0=\frac{\partial U^i_\Lambda}{\partial x_n}(x_k)=	\frac{\partial U^i_\Lambda}{\partial x_n}(y_k)+\int_{y_{k,1}}^{x_{k,1}}\frac{\partial^2 U^i_\Lambda}{\partial x_1\partial x_n}(s,x')ds>0$$
	which is a contradiction.
\end{enumerate}
	
In summary, the above two cases would not happen, and hence $\Lambda=\Lambda_0$.
	
{\bf Step 3}: Conclusions.
Now $\Lambda=\Lambda_0$, then $U^i_\Lambda\leq 0$ by the continuity of $u^i$, and $\frac{\partial u^i}{\partial x_1}<0$ in $\{x\in \Omega\ |\ x_1<\Lambda_0\}$ by \cref{lem:SMPFMA} on $\Sigma_\lambda,\ \forall\lambda<\Lambda$ with a covering argument.

For the second part assertion in the theorem, noting that \eqref{eq:EIB} also holds on $\Sigma_\Lambda$,  hence $\frac{\partial u^i}{\partial x_1}(\Lambda_0,x')=0$ for some $x\in \{x_1=\Lambda_0\}$ and \cref{lem:SMPFMA} shows that $U^i_\Lambda\equiv0$ in $\Sigma_\Lambda$, and then we finish the proof of the whole theorem.

\end{proof3}

\section{Application}\label{sec:app}
\noindent

We consider the following elliptic system coupled by Monge-Amp\`ere equations: 
\begin{equation}\label{app1}
\left\{\begin{array}{rl}
&\operatorname{det} D^2u^1=(-u^2)^{\alpha}\text{ in }\Omega,\\
&\operatorname{det} D^2u^2=(-u^1)^{\beta}\text{ in }\Omega,\\
&u^1<0, u^2<0\hspace{0.85cm} \text{ in }\Omega,\\
&u^1=u^2=0\hspace{1.2cm} \text{ on }\partial\Omega.
\end{array}
\right.
\end{equation} 

The existence of \eqref{app1} have been obtained in \cite{zhang_power-type_2015}.
\begin{thm}
\label{thm:MAZQ}
Let $\Omega=B_1(0)\subset\br^n$,$\alpha>0,\beta>0$, then
\begin{enumerate}[label=$(\roman{enumi})$]
\item if $\alpha\beta\neq n^2$, \eqref{app1} have at least one radially symmetric convex solution;
\item if $\alpha\beta< n^2$, \eqref{app1} have exactly one radially symmetric convex solution;
\item if $\alpha\beta= n^2$, \eqref{app1} have no radially symmetric convex solution.
\end{enumerate}
\end{thm}

We apply \cref{thm:ball} to get the following uniqueness theorem.

\begin{thm}
Let $\alpha>0$, $\beta>0$, $\alpha\beta<n^2$ and  $\Omega=\{x\in \mathbb{R}^n:|x|<1\}$ with $n\geq 2$. Then \eqref{app1} admits unique convex solutions ${\bf u}=(u^1,u^2)$, which must be radially symmetric and strictly increasing with respect to $0$.
\end{thm}
\begin{proof}
On the one hand, according to \cref{thm:MAZQ}, \eqref{app1} has a unique radial convex solution ${\bf u}=(u^1,u^2)$. On the other hand, since \eqref{app1} satisfies the condition of \cref{thm:ball}, it follows that all convex solutions ${\bf u}=(u^1,u^2)$ to system \eqref{app1} must be radially symmetric with respect to $0$. Therefore we have showed that there is only one group of convex solutions ${\bf u}=(u^1,u^2)$ to system \eqref{app1}, which must be radial and strictly increasing.
\end{proof}


\bibliographystyle{abbrv}
\bibliography{ref}

\end{document}